\documentclass{amsart}
\usepackage{amssymb, enumerate, xspace, graphicx}
\usepackage[all]{xy}
\usepackage{pgf}
\usepackage{scrextend}
\SelectTips{cm}{}
\usepackage{url}
\usepackage{tikz}
\usepackage[T1]{fontenc}
\usepackage[final]{microtype}
\usetikzlibrary{arrows,matrix,positioning}
\usepackage{tikz-cd}
\usepackage{filecontents}

\usepackage{mathtools}

\numberwithin{equation}{section}

\numberwithin{subsection}{section}

\allowdisplaybreaks[1]

\newlength{\myarrowsize} 
    \newlength{\myoldlinewidth}
\pgfarrowsdeclare{myto}{myto}{
        \pgfsetdash{}{0pt} 
        \pgfsetbeveljoin 
        \pgfsetroundcap 
        \setlength{\myarrowsize}{0.6pt}
        \addtolength{\myarrowsize}{.5\pgflinewidth}
        \pgfarrowsleftextend{-4\myarrowsize-.5\pgflinewidth} 
        \pgfarrowsrightextend{.7\pgflinewidth}
    }{
        \setlength{\myarrowsize}{0.6pt} 
        \addtolength{\myarrowsize}{.5\pgflinewidth}  
        \setlength{\myoldlinewidth}{\pgflinewidth}
        \pgfsetroundjoin
        \pgfsetlinewidth{0.0001pt}
        \pgfpathmoveto{\pgfpoint{0.43\myarrowsize}{0}}
        \pgfpatharc{0}{70}{0.14\myarrowsize}
        \pgfpatharc{-110}{-169.5}{4\myarrowsize}
        \pgfpatharc{10.5}{189}{0.25\myarrowsize and 0.12\myarrowsize}
        \pgfpatharc{-170}{-119.5}{4.48\myarrowsize}
        \pgfpathmoveto{\pgfpoint{0.43\myarrowsize}{0}}
        \pgfpatharc{0}{-70}{0.14\myarrowsize}
        \pgfpatharc{110}{169.5}{4\myarrowsize}
        \pgfpatharc{-10.5}{-189}{0.25\myarrowsize and 0.12\myarrowsize}
        \pgfpatharc{170}{119.5}{4.48\myarrowsize}
        \pgfpathclose
        \pgfsetstrokeopacity{0.25}
        \pgfusepathqfillstroke
    }

\pgfarrowsdeclare{mydashto}{mydashto}{
        \pgfsetdash{{3pt}{3pt}}{0pt} 
        \pgfsetbeveljoin 
        \pgfsetroundcap 
        \setlength{\myarrowsize}{0.6pt}
        \addtolength{\myarrowsize}{.5\pgflinewidth}
        \pgfarrowsleftextend{-4\myarrowsize-.5\pgflinewidth} 
        \pgfarrowsrightextend{.7\pgflinewidth}
    }{
        \setlength{\myarrowsize}{0.6pt} 
        \addtolength{\myarrowsize}{.5\pgflinewidth}  
        \setlength{\myoldlinewidth}{\pgflinewidth}
        \pgfsetroundjoin
        \pgfsetlinewidth{0.0001pt}
        \pgfpathmoveto{\pgfpoint{0.43\myarrowsize}{0}}
        \pgfpatharc{0}{70}{0.14\myarrowsize}
        \pgfpatharc{-110}{-169.5}{4\myarrowsize}
        \pgfpatharc{10.5}{189}{0.25\myarrowsize and 0.12\myarrowsize}
        \pgfpatharc{-170}{-119.5}{4.48\myarrowsize}
        \pgfpathmoveto{\pgfpoint{0.43\myarrowsize}{0}}
        \pgfpatharc{0}{-70}{0.14\myarrowsize}
        \pgfpatharc{110}{169.5}{4\myarrowsize}
        \pgfpatharc{-10.5}{-189}{0.25\myarrowsize and 0.12\myarrowsize}
        \pgfpatharc{170}{119.5}{4.48\myarrowsize}
        \pgfpathclose
        \pgfsetstrokeopacity{0.25}
        \pgfusepathqfillstroke
    }

\newtheorem*{namedtheorem}{\theoremname}
\newcommand{\theoremname}{testing}

\newtheorem{theo}{Theorem}[section]
\newtheorem{prop}[theo]{Proposition}
\newtheorem{prop-def}[theo]
{Proposition-Definition}
\newtheorem{coro}[theo]{Corollary}
\newtheorem{lemm}[theo]{Lemma}

\theoremstyle{definition}
\newtheorem{defi}[theo]{Definition}

\newtheorem{exam}[theo]{Example}

\newtheorem{rema}[theo]{Remark}

\theoremstyle{remark}

\DeclareMathOperator{\Lie}{Lie}

\newcommand{\reftext}{}

\newcommand\mto{\ifinner\mapsto\else\longmapsto\fi}

\def\displaytimes_#1{\mathrel{\mathop{\times}\limits_{#1}}}

\def\displayotimes_#1{\mathrel{\mathop{\bigotimes}\limits_{#1}}}

\begin{document}


\title[Essential dimension  of group schemes over a local scheme]{Essential dimension of group schemes\\ over a local scheme}

\author{Dajano Tossici}

\address{Institut de Math\'ematiques de Bordeaux, 351 Cours de la Liberation\\
33 405 Talence \\ France}
\email[Tossici]{Dajano.Tossici@math.u-bordeaux.fr}
\maketitle
%
%
%
%
%
%
%
\begin{abstract}             
In this paper we develop the theory of essential dimension of group
schemes over an integral base. Shortly we concentrate over a local base.
As a consequence of our theory we give a result of invariance of the
essential dimension over a field. The case of group schemes over a
discrete valuation ring is discussed.
\end{abstract}
%
%
%
%




\section{Introduction}

The notion of essential dimension of a finite group over a field $k$ was
introduced by Buhler and Reichstein \cite{BR}. It was later extended
to various contexts. First Reichstein generalized it to linear algebraic
groups \cite{Re} in characteristic zero; afterwards Merkurjev gave
a general definition for functors from the category of extension fields
of the base field $k$ to the category of sets \cite{BF}. In
particular one can consider the essential dimension of group schemes
over a field (see Definition \ref{def:edusuale}).

In this paper we would like to extend the notion of essential dimension
of a group scheme over a base scheme more general than a field.

If $G$ is an affine flat group scheme of locally finite presentation
over $S$, a $G$-torsor over $X$ is an $S$-scheme $Y$ with a left
$G$-action by $X$-automorphisms and a faithfully flat and locally of
finite presentation morphism $Y\to X$ over $S$ such that the map
$G\times_{S} Y\to Y\times_{X} Y$ given by $(g,y)\mapsto (gy,y)$ is an
isomorphism. We recall that isomorphism classes of $G$-torsors over $X$
are classified by the pointed set $H^{1}(X,G)$ \cite[III,
Theorem~4.3]{Mi}. If $G$ is commutative, then $\operatorname{H}^{1}(X,
G)$ is a group, and coincides with the cohomology group of $G$ in the
fppf topology.

\begin{defi}
\label{def:edusuale}
Let $G$ be an affine group scheme of finite type over a field $k$. Let
$k\subseteq K$ be an extension field and $[\xi ]\in \operatorname{H}
^{1}(\operatorname{Spec}(K),G)$ be the class of a $G$-torsor
$\xi $. Then the essential dimension of $\xi $ over $k$, which we denote
by $\operatorname{ed}_{k}\xi $, is the smallest non-negative integer
$n$ such that
\begin{itemize}
\item[(i)] there exists a subfield $L$ of $K$ containing $k$, with
$\mathop{\mathrm{tr}\, \mathrm{deg}}\nolimits (L/k) = n$,
\item[(ii)] such that $[\xi ]$ is in the image of the morphism
\begin{equation*}
\operatorname{H}^{1}(\operatorname{Spec}(L),G)\longrightarrow
\operatorname{H}^{1}(\operatorname{Spec}(K),G).
\end{equation*}
\end{itemize}
The essential dimension of $G$ over $k$, which we denote by
$\operatorname{ed}_{k} G$, is the supremum of $\operatorname{ed}_{k}
\xi $, where $K/k$ ranges through all the extension of $K$, and
$\xi $ ranges through all the $G$-torsors over $\operatorname{Spec}(K)$.

\end{defi}

Moreover there is another possible definition, more geometric, for the
essential dimension of group schemes. See \cite[Definition 2.5]{BR} and
\cite[Definition 6.8]{BF}.

\begin{defi}
\label{def:classicalessentialdimension}
Let $G$ be an affine group scheme of finite type over a field $k$. The
\textit{essential dimension} of a $G$-torsor $f:Y\to X$ is the smallest
dimension of a scheme $X'$ over $k$ such that
\begin{itemize}
\item[(i)] there exists a $G$-torsor $f':Y'\to X'$,
\item[(ii)]and a commutative diagram of $k$-rational maps
\begin{equation*}
\vcenter{\xymatrix{Y\ar@{-->}^g[r]\ar^f[d]& Y'\ar^{f'}[d]\\ X\ar@{-->}^{h}[r]& X'}},
\end{equation*}
\end{itemize}
with $g$ dominant and $G$-equivariant. It will be denoted by
$\operatorname{ed}_{S} f$. The \textit{essential dimension} of $G$ over
$k$ is $\sup_{f} \operatorname{ed}_{k} f$ where $f$ varies between all
$G$-torsors. It will be denoted by $\operatorname{ed}_{k} G$.
\end{defi}

We are going to generalize this second definition over a base more
general than a field. At a very first sight, over a field, the new
definition could look slightly different from \reftext{Definition~\ref{def:classicalessentialdimension}}. However we will prove that
they are equivalent.

Very soon we specialize to the case that the base is local. And we
generalize some standard basic results true over a field. For instance
we prove the existence, under some conditions on $G$, of a classifying
torsor whose essential dimension gives the essential dimension of the
group scheme $G$.

As a natural application of this theory we prove the following result.

{
\renewcommand{\thetheo}{4.1}
\begin{theo}
Let $G$ be an affine faithfully flat group scheme of finite presentation over an integral locally noetherian scheme
$S$. Then there exists a non-empty open subscheme $U$ of $S$ such that
for any $s\in U$, with residue field $k(s)$,
\begin{equation*}
\operatorname{ed}_{k(s)}(G_{k(s)})\le \operatorname{ed}_{k(S)}(G_{k(S)})=
\operatorname{ed}_{\mathcal{O}_{S,s}}G_{\mathcal{O}_{S,s}}.
\end{equation*}
\end{theo}
\addtocounter{theo}{-1} }

In particular we obtain the following corollary.

{
\renewcommand{\thetheo}{4.4}
\begin{coro}
Let $k$ be a field, let $X$ be an integral scheme of finite type over
$k$ with fraction field $k(X)$ and let $G$ be an affine group scheme of
finite type over $k$. There exists a non-empty open subscheme $U$ of
$X$ such that for any $x\in U$, with residue field $k(x)$,
\begin{equation*}
\operatorname{ed}_{k(x)}(G_{k(x)})\le \operatorname{ed}_{k(X)}(G_{k(X)})=
\operatorname{ed}_{\mathcal{O}_{X,x}}G_{\mathcal{O}_{X,x}}.
\end{equation*}
In particular if the set of $k$-rational points of $X$ is Zariski dense
then
\begin{equation*}
\operatorname{ed}_{k}(G)=\operatorname{ed}_{k(X)}(G_{k(X)}).
\end{equation*}
\end{coro}
\addtocounter{theo}{-1}}%

If $X$ is the affine line the last part of the
Corollary is the \textit{homotopy invariance} theorem of Berhui and Favi
\cite[Theorem 8.4]{BF}. Of course by induction it works over
$\mathbb{A}^{n}_{k}$. It seems to us that the strategy of that proof can
not be immediately generalized to the general case because of technical
Lemma \cite[Lemma 8.3]{BF}.

As A. Vistoli pointed to us, the last part of the Corollary was also
proven in the unpublished result \cite[Proposition 2.14]{BRV} in the
case $X=\mathbb{A}^{n}_{k}$ or $k$ algebraically closed. In this case
it seems that their argument can be generalized to (geometrically
integral) varieties with the set of $k$-rational points which is Zariski
dense.

It would be interesting to know if the converse of the last part of the
\reftext{Corollary~\ref{Thm:homotopytheorem}} is true, more precisely

\textit{(Q) if $X$ is an integral scheme of finite type over $k$ and
$ \operatorname{ed}_{k}(G)=\operatorname{ed}_{k(X)}(G_{k(X)}) $ for any
affine group scheme $G$ of finite type over $k$, is the set of
$k$-rational point of $X$ Zariski dense?}

Some considerations are given at the end of section \S \ref{sec4}. At the moment
the question is open, except for finite fields (proven in
\cite[Proposition 3.6 and Lemma 4.5]{NDT}).

\section*{Acknowledgments}

\subsection*{Acknowledgements} I would like to thank Q.~Liu,
C.~Pepin,  M.~Romagny, J.~ Tong and A.~Vistoli   for very useful comments and conversations.  Finally I also thank the referee for useful commments.
I have been partially supported by the project ANR-10-JCJC 0107 from the
Agence Nationale de la Recherche.


\section{$S$-rational maps}\label{sec2}

In this section we recall some notions about rational morphisms. Here
$S$ will denote a general scheme. Starting from the next section we will
add some hypotheses on it. We will put some details since definitions
we give here are slightly different from classical references such as
\cite[\S 2.5]{BLR} or \cite[\S 20]{EGAIV4}. We will point out the
differences later on. We begin with the notion of schematically
dominant.
\begin{defi}
Let $f:X\to Y$ be a morphism of schemes. We say that $f$ is
\textit{schematically dominant} if $f^{\#}:\mathcal{O}_{Y}\to f_{*}
\mathcal{O}_{X}$ is injective. We say that $f:X\to Y$ is
\textit{schematically dense} if it is schematically dominant and an open
immersion.
\end{defi}
\begin{rema}
If $Y$ is reduced then we recover the usual definitions of dominant and
dense \cite[Proposition 11.10.4]{EGAIV3}. Without other assumptions the
above definition \textit{does not} mean that a morphism is schematically
dominant if and only if the schematic image is $Y$. This is true if the
morphism is quasi-compact (see \cite[Proposition 10.30]{GW}).
\end{rema}
We also have the relative version.
\begin{defi}
Let $f:X\to Y$ be a morphism of $S$-schemes. We say that $f$ is
$S$-dominant if, for any $s\in S$, the morphism $X_{s}\to Y_{s}$ is
schematically dominant, where
$Y_{s}:=Y\times_{S}\operatorname{Spec}(k(s))$ and $X_{s}:=X\times_{S}
\operatorname{Spec}(k(s))$ are the fibers over $S$. If moreover $f$ is
an open immersion we say that $X$ is $S$-\textit{dense} in $Y$.

We say that $f$ is $S$-\textit{universally dominant} if it is
$S'$-dominant under any base change $S'\to S$. In the case of an open
immersion we say that $f$ is $S$-\textit{universally dense}.
\end{defi}

We remark that $S$-universally schematically dominant implies
$S$-schematically dominant. Under some mild conditions the converse is
true.
\begin{prop}
\label{univdomin=schemdomin}
Let $f:X\to Y$ be a morphism of $S$-schemes with $X$ flat over $S$. In
any of these two situations
\begin{itemize}
\item[(i)] $Y$ locally noetherian,
\item[(ii)] $f$ is open immersion and $Y$ flat locally of finite
presentation over $S$,
\end{itemize}
then $f$ is $S$-dominant if and only if it is $S$-universally dominant.
\end{prop}
\begin{rema}
\label{rem:univdom}
We remark that from the Proposition it follows that, under the same
hypotheses, $f$ is $S$-schematically dominant if and only if
$f_{T}$ is $T$-schematically dominant for any base change $T\to S$.
\end{rema}

\begin{proof}
\cite[Th\'{e}or\`{e}me 11.10.9, Proposition 11.10.10]{EGAIV3}
\end{proof}

Let $X$ and $Y$ be two $S$-schemes. Let $U$ and $U'$ be two $S$-dense
open subschemes of $X$. If we have two morphisms $f:U\to Y$ and
$f':U'\to Y$ we say that they are \textit{equivalent} if there exists
an open subscheme $V\subseteq U'\cap U$, which is $S$-dense in $X$, such
that $f$ and $f'$ coincide over $V$. One easily verifies that it is an
equivalence relation.
\begin{defi}
An $S$-rational map between two $S$-schemes $X$ and $Y$ is the
equivalence class of an $S$-morphism $f:U\to Y$ where $U$ is $S$-dense
in $X$. An $S$-rational morphism is denoted by $f:X\dashrightarrow Y$
and for any $U$ as above we say that $f$ is defined over $U$.
\end{defi}

The definition here is stronger than EGA's definition of $S$-pseudo
morphism \cite[20.2.1]{EGAIV4}. In fact there it is not required that
the open subscheme is schematically dense on each fiber of $X$. This difference
will be very important in the definition of compressions
(\S \ref{sectio:compressions}), in order to have a theory of essential
dimension over a general base which is compatible with the theory we
obtain when we restrict to a point. While in \cite[\S 2.5]{BLR} the
definition of rational morphisms is the same except the fact that all
schemes considered are $S$-smooth and so $S$-dense means just Zariski
dense on the fibers.

\begin{defi}
An $S$-rational map $f: X\dashrightarrow Y$ is $S$-dominant if it can
be represented by an \textit{$S$-dominant} morphism.
\end{defi}

In fact, by the following lemma, the above definition can be restated
saying that $f$ is $S$-dominant if any of its representatives is
$S$-dominant.

\begin{lemm}
Let $f: X\dashrightarrow Y$ be an $S$-rational map. Let us consider two
representatives $f_{1}:U_{1}\to Y$ and $f_{2}:U_{2}\to Y$. Then
$f_{1}$ is $S$-dominant if and only if $f_{2}$ is $S$-dominant.
\end{lemm}
\begin{proof}
By symmetry it is sufficient to prove just one implication. Moreover,
by definition, we can reduce to the case $S$ is the spectrum of a field
and in this case $S$-dense is simply schematically dense. Now let
$V$ be an open subscheme $V\subseteq U_{1}\cap U_{2}$ and schematically
dense in $X$ such that $f_{1}$ and $f_{2}$ coincide over $V$, and we
call $g$ the restriction. Let us suppose that $f_{1}$ is schematically
dominant. Then the morphism $\mathcal{O}_{Y}\to {(f_{1})}_{*}
\mathcal{O}_{U_{1}}$ is injective. Since $V$ is schematically dominant
in $X$ it is also schematically dominant in~$U_{1}$. So, using the
following diagram, {one can} easily see that $g$ is schematically dominant and
therefore $f_{2}$ is schematically dominant.
\begin{equation*}
\vcenter{\xymatrix{  &(f_1)_*\mathcal {O}_{U_1}\ar[rd]& \\
\mathcal {O}_Y\ar[ru]\ar[rd]& & g_*\mathcal {O}_V\\
& (f_2)_*\mathcal {O}_{U_2}\ar[ru]&}}\qedhere
\end{equation*}
\end{proof}

\begin{defi}
Let $f:X\dashrightarrow Y$ be an $S$-rational map. We say that it is
$S$-\textit{birational} if it is $S$-dominant and there exists a
representative $f_{1}:U\to Y$ which is an open immersion. We will say
that $X$ and $Y$ are $S$-birational if there exists an $S$-birational
map between $X$ and $Y$.
\end{defi}

\begin{defi}
Let $S$ be a scheme, $f:X\dashrightarrow Y$ and $g:Y\dashrightarrow Z$
two $S$-rational maps. We call $g\circ f:X\dashrightarrow Z$, if it
exists, the rational map represented by $g\circ {f_{U}}$, where $U$ is
an $S$-dense open subscheme where $f$ is defined and such that $g$ is
defined over $f(U)$.
\end{defi}

In general it is not possible to define the composition of two
$S$-rational maps, even if they are $S$-schematically dominant and the
schemes are irreducible. This is possible if we use the classical
definition of dominant. Here is an example where the composition does
not work.

\begin{exam}
Let $k$ be a field. Let us consider
$X=\operatorname{Spec}(k[x,y]/(xy, x^{2}))$, $Y=
\operatorname{Spec}(k[x,y]/\allowbreak(x^{2}))$ and $Z$ equal to $Y$ minus the
origin. Then let $f:X\to Y$ be the natural inclusion and let
$g:Y\dashrightarrow Z$ the birational morphism induced by the identity
over $Z\subseteq Y$. Then the composition $g\circ f$ is defined over
the open subscheme $X$ minus the origin. But this open subscheme, which is the
maximal where $g\circ f$ is defined, is not schematically dense since
the embedded point $(x,y)$ does not belong to it.
\end{exam}

However we can define the composition of $S$-rational maps in some
cases.
\begin{lemm}
Let $S$ be a scheme, $f:X\dashrightarrow Y$ and $g:Y\dashrightarrow Z$
two $S$-rational maps. In the following cases the composition
$g\circ f$ exists.
\begin{itemize}
\item[(i)] $g$ is a morphism,
\item[(ii)] $f$ is a flat locally of finite presentation morphism and
$Y$ is locally {noetherian},
\item[(iii)] $X\to S$ has integral fibers and $f$ is $S$-dominant.
\end{itemize}
\end{lemm}
\begin{proof}
The first case is clear. For the other two situations take an $S$-dense
$V$ of $Y$ where $g$ is defined. We will prove that $U:=f^{-1}(V)$ is
$S$-dense in $X$ and so the composition exists. Clearly we can suppose
that $S$ is a point. For (ii) we remark that since $f$ is flat locally
of finite presentation then $f$ is an open map. So $f(X)$ is open. Then
it intersects~$V$, which is schematically dense. Therefore $f^{-1}(V)$
is non-empty. Now since $f$ is flat then $f^{-1}(V)$ is schematically
dense in $X$, by \cite[Lemma 28.24.13, Tag 081H]{stacksproject}, since
any open subscheme of a locally  noetherian  scheme is retrocompact.

For (iii) we observe that $f^{-1}(V)$ is non-empty since $f$ is
schematically dominant. Now if $f^{-1}(V)$ is a non-empty open set of
$X$ then it is dense, since $X$ is irreducible. But then it is also
schematically dense since $X$ is reduced.
\end{proof}

It is easy to see that the composition is well defined, i.e. does not
depend on the representative of $f$.

\begin{lemm}
\label{lem:fiberproduct}
Let $S$ be a scheme and let $f:Y\to Z$ and $g:W\to Z$ be morphisms of
$S$-schemes. For any $S$-rational maps $h_{1}:T\dashrightarrow Y$ and
$h_{2}:T\dashrightarrow W$ such that $f\circ h_{1}$ is equal to
$g\circ h_{2}$ as $S$-rational maps then there exists a unique
$S$-rational map $h:T\dashrightarrow Y\times_{Z} W$ such that
$p_{Y}\circ h=h_{1}$ and $p_{W}\circ h=h_{2}$ where $p_{Y}$ and
$p_{W}$ are the projections over $Y$ and $W$.
\end{lemm}
\begin{proof}
Easy to prove using the universal property of cartesian product over an
$S$-dense open subscheme where $f\circ h_{1}$ and $g\circ h_{2}$ are
equal.
\end{proof}

\begin{lemm}
\label{lem:basechangerationalmaps}
Let $f:X\dashrightarrow Y$ be an $S$-rational map with $X$ flat locally
of finite presentation over $S$. For any morphism $T\to S$, we have a
$T$-rational map $f_{T}:{X_{T}\dashrightarrow Y_{T}}$ obtained by base
change.
\end{lemm}

\begin{proof}
If $f$ is defined over an $S$-dense $U$ then $U_{T}$ is
$T$-schematically dense in $X_{T}$ by
\reftext{Proposition~\ref{univdomin=schemdomin}}. Therefore, using
\reftext{Lemma~\ref{lem:fiberproduct}} we have a $T$-rational map
$f_{T}:{X_{T}\dashrightarrow Y_{T}}$.
\end{proof}

\section{Definitions and general results}\label{sectio:compressions}

In the following $S$ will be an integral locally noetherian scheme. And
till the end of the paper, if not differently specified, for \textit{group scheme}  we will mean an
affine faithfully flat group scheme of finite presentation over the
base. Let $f:X\to S$ be a faithfully flat morphism of locally finite
type. If $\eta $ is the generic point of $S$ we call $\dim (f^{-1}
\eta )$ the \textit{relative dimension} of $X$ over $S$ and we denote
it by $\dim_{S} X$.

If $X$ is also irreducible then $f:X\to S$ is equidimensional, i.e. for
any $x\in X$, $\dim (f^{-1}\eta )=\dim_{x}f^{-1}(f(x))$ (see
\cite[Corollaire 6.1.1, Proposition 13.2.3]{EGAIV2}).

For any scheme $T$ we call $\mathfrak{C}_{T}$ the full subcategory of
$(\mathrm{Sch}/T)$ given by faithfully flat schemes $X$ of locally
finite presentation over $T$ with geometrically integral fibers.

\begin{lemm}
\label{lemm:integralfibersimpliesintegral}
Let $f:X\to T$ be a flat morphism.
\begin{itemize}
\item[(i)] If $f$ is of locally finite presentation, $T$ irreducible and
any fiber is irreducible then $X$ is irreducible.
\item[(ii)] If $X$ and $T$ are locally noetherian, $T$ is reduced and
any fiber of $f$ is reduced then $X$ is reduced.
\end{itemize}
\end{lemm}
\begin{proof}
(i) Since $f$ is flat of locally finite presentation then $f$ is open.
Now let $U$ and $V$ be two open sets of $Y$. Then $f(U)$ and
$f(V)$ are open since $f$ is open. Moreover their intersection is
non-empty since $T$ is irreducible. Let $t$ be a point of this
intersection. So $U$ and $V$ intersect the fiber over $t$. Since any
fiber of $f$ is irreducible then there is point over $t$ contained in
$U$ and $V$. So $X$ is irreducible.

(ii) This is \cite[Corollaire 3.3.5]{EGAIV2}.
\end{proof}

By the previous Lemma any object of $\mathfrak{C}_{T}$ with $T$ integral
is integral.

And if $T'$ is another scheme with a morphism $\pi :T'\to T$ and $X$ is
an object of $\mathfrak{C}_{T}$ then $X_{T'}:=X\times_{T} T'$ is an
object of $\mathfrak{C}_{T'}$. In fact if $t'$ is a point of $T$ then
$X_{t'}$ is isomorphic to $X_{t}\times_{\operatorname{Spec}(k(t))}
\operatorname{Spec}(k(t'))$ where $\pi (t')=t$. So $X_{t'}$ is
geometrically integral.

\begin{defi}
Let $G$ be a group scheme over $S$. Let $f:Y\to X$ and $f':Y' \to X'$
be two $G$-torsors with $X$ and $X'$ objects of $\mathfrak{C}_{S}$. We
say that $f'$ is an $S$-\textit{weak compression} of $f$ if there exists
a diagram over $S$
\begin{equation*}
\vcenter{\xymatrix{Y\ar@{-->}^g[r]\ar^f[d]& Y'\ar^{f'}[d]\\ X\ar@{-->}^{h}[r]& X'}},
\end{equation*}
which is commutative (i.e. $f'\circ g$ and $h\circ f$ are equal as
$S$-rational maps), where $g$ and $h$ are $S$-rational and $g$ is
$G$-equivariant (i.e. there exists an open $S$-dense subscheme $U$ of
$Y$ stable by $G$ such that $g_{U}:U\to X$ represents $g$ and it is
$G$-equivariant).

We say that $f'$ is an \textit{$S$-compression} of $f$ if moreover
$g$ is an $S$-dominant map. And we say that a weak $S$-compression
(resp. $S$-compression) $f'$ is \textit{defined everywhere} if $g$ and
$h$ are morphisms.
\end{defi}

We have the following easy result.
\begin{lemm}
\label{lem:dominantcompressions}
Let $X,X'$ be objects of $\mathfrak{C}_{S}$. If $f':Y'\to X'$ is a weak
$S$-compression of $f:Y\to X$ then there exists an $S$-dense open
subscheme $U$ of $X$ such that $f':Y'\to {X'}$ is a defined everywhere
weak $S$-compression of $f_{U}:Y_{U}\to U$.

If $f'$ is an $S$-compression then for any $S$-dense open subscheme
$U'$ of $X'$ there exists an $S$-dense open subscheme $U$ of $X$ such
that $f'_{U'}:Y'_{U'}\to {U'}$ is a defined everywhere $S$-compression
of $f_{U}:Y_{U}\to U$.
\end{lemm}
\begin{proof}
We use notation of the definition. Take an $S$-dense open subscheme
$W$ in $Y$ which is $G$-stable and such that $g$ is defined over $W$.
Then $W\to f(W)$ is a $G$-torsor and $f(W)$ is an $S$-dense open
subscheme. The induced morphism $f(W)\to X'$ clearly represents $h$. If
we set $f(W)=U$ we have that $f'$ is a weak compression of $f_{U}$.

The last part follows remarking that if $U'$ is an $S$-dense open
subscheme of $Y$ then $f_{U'}:Y'_{U'}\to {U'}$ is a weak $S$-compression
of $f$ since $g$ is dominant.
\end{proof}

\begin{defi}
\label{def:essentialdimension}
The \textit{essential dimension} of a $G$-torsor $f:Y\to X$ is the
smallest relative dimension of $X'$ over $S$ in a weak compression
$f':Y'\to X'$ of $f$, where $X'$ is an object of $\mathfrak{C}_{S}$ and
it will be denoted by $\operatorname{ed}_{S} f$. The \textit{essential
dimension} of $G$ over $S$ is $\sup_{f} \operatorname{ed}_{S} f$ where
$f$ varies between all $G$-torsors over objects of $\mathfrak{C}_{S}$.
It will be denoted by $\operatorname{ed}_{S} G$.
\end{defi}

We recall that, apparently, the above definition is different, in the
case $S$ is a field, from the usual definition which uses compressions
(\reftext{Definition~\ref{def:classicalessentialdimension}}). There are three
differences. First usually one considers compressions instead of weak
compressions. Over a field $k$ this is not a problem: if we have a
$k$-weak compression
\begin{equation*}
\vcenter{\xymatrix{Y\ar@{-->}^g[r]\ar^f[d]& Y'\ar^{f'}[d]\\ X\ar@{-->}^{h}[r]& X'}},
\end{equation*}
then we have that $f'_{Z}:Y'_{Z}\to Z$ is a $k$-compression of $f$,
where $Z$ is the schematic image of $h$. In fact it is clear that
$h:X\dashrightarrow Z$ is $k$-dominant. Since $f'_{Z}$ is faithfully
flat then also $g:Y\dashrightarrow Y'_{Z}$ is $k$-dominant.

The second difference is the fact that we are supposing that the scheme
$X'$ is \textit{geometrically} integral. Thirdly we take locally finite
presentation schemes, but this does not cause problems since the
essential dimension of an algebraic affine group scheme is finite.

If $S$ is the spectrum of a field then by \cite[Lemma~6.11 and
Remark~6.12]{BF} the classical definition using compressions
\cite[Definition~6.8]{BF} of essential dimension is equivalent to the
functorial definition. We will prove later in
\reftext{Proposition~\ref{Prop:edoverafield}} that over a field they
are both equivalent to our definition.

\begin{lemm}
\label{lem:compositionofcompressions}
Let $f:Y\to X$, $f':Y'\to X'$ and $f'':Y''\to X''$ be three $G$-torsors
over $S$ with $X,X',X''$ objects of $\mathfrak{C}_{S}$.
\begin{itemize}
\item[(i)] If $f''$ is a defined everywhere weak $S$-compression of
$f'$ and $f'$ is a weak $S$-com\-pression of $f$ then $f''$ is a weak
$S$-compression of $f$.
\item[(ii)] If $f''$ is a weak $S$-compression of $f'$ and $f'$ is an
$S$-compression of $f$ then $f''$ is a weak $S$-compression of $f$.
\item[(iii)] If $f''$ is an $S$-compression of $f'$ and $f'$ is an
$S$-compression of $f$ then $f''$ is an $S$-compression of $f$.
\end{itemize}

\end{lemm}
\begin{proof}
The proof of $(i)$ is immediate, since in this case the composition of
the involved rational maps is well defined.

Now since $f''$ is a weak $S$-compression of $f'$ then, by
\reftext{Lemma~\ref{lem:dominantcompressions}} there exists an
$S$-dense open subscheme $U'$ of $X'$ such that $f''$ is a defined
everywhere weak $S$-compression of $f'_{U'}$. Therefore, by the last
part of \reftext{Lemma~\ref{lem:dominantcompressions}}, there exists an
$S$-dense open subscheme $U$ of $X$ such that $f'_{U'}$ is an
$S$-defined everywhere compression of $f_{U}$. Applying $(i)$ we
obtain~$(\mathit{ii})$. Moreover we also obtain $(\mathit{iii})$ since composition of
dominant maps is dominant.
\end{proof}

\begin{lemm}
\label{lem:basechange}
Let $f:Y\to X$ be a $G$-torsor over $S$ with $X$ an object of
$\mathfrak{C}_{S}$ which is integral. For any morphism $q:T\to S$, with
$T$ integral locally noetherian, then
\begin{equation*}
\operatorname{ed}_{S} f\ge \operatorname{ed}_{T} f_{T}.
\end{equation*}
\end{lemm}

\begin{proof}
We can clearly suppose that $\operatorname{ed}_{S} f$ is finite. Let us
consider a weak $S$-compression
\begin{equation*}
\vcenter{\xymatrix{Y\ar@{-->}^g[r]\ar^f[d]& Y'\ar^{f'}[d]\\ X\ar@{-->}^{h}[r]& X'}}.
\end{equation*}
So, using \reftext{Lemma~\ref{lem:basechangerationalmaps}}, we obtain, by base
change over $T$, a weak $T$-compression
\begin{equation*}
\vcenter{\xymatrix{Y_T\ar@{-->}^{g_T}[r]\ar^{f_T}[d]&
Y'_T\ar^{f'_T}[d]\\ X_T\ar@{-->}^{h_T}[r]& X'_T}}.
\end{equation*}
So the above diagram gives a weak $T$-compression. Let us take $X'$ such
that $\operatorname{ed}_{S} f=\dim_{S} X'$. By \reftext{Lemma~\ref{lemm:integralfibersimpliesintegral}}, $X'$ is irreducible so,
as remarked at the beginning of the section, the morphism $f:X'\to S$
is equidimensional. We have now to compute $\dim_{T} X'_{T}=\dim \pi
_{T}^{-1} \xi $ where $\pi : X'_{T}\to T$ is the structural morphism and
$\xi $ is the generic point of $T$. If $s=q(\xi )$ then, as schemes,
$\pi_{T}^{-1} \xi $ is isomorphic to $X'_{s}\times_{s} \xi $. So
$\dim_{T} X'_{T}=\dim X'_{s}$ since the dimension of a finite type
scheme over a field does not change when extending the field. Since
$X'$ is equidimensional we also have $\dim X'_{s}=\dim_{S} X'$. So
\begin{equation*}
\operatorname{ed}_{S} f=\dim_{S} X'=\dim_{T} X'_{T}\ge
\operatorname{ed}_{T} f_{T}.\qedhere
\end{equation*}
\end{proof}

\begin{defi}
Let $f:Y\to X$ be a $G$-torsor over $S$ with $X$ an object of
$\mathfrak{C}_{S}$. We will say that it is $S$-\textit{classifying} if
for any $S$-dense open subscheme $U$ of $X$ and for any $G$-torsor
$f':Y'\to X'$ over $S$, where $X'$ is an object of $\mathfrak{C}_{S}$
with positive $S$-dimension, then $f_{U}:Y_{U}\to U$ is a weak
compression of $f'$.
\end{defi}

\begin{lemm}
The compression of an $S$-classifying $G$-torsor is $S$-classifying.
\end{lemm}
\begin{proof}
Let $f:Y\to X$ be an $S$-classifying $G$-torsor, let $f'':Y''\to X''$
be a compression of $f$ and let $f':Y'\to X'$ be any $G$-torsor over
$S$ with $X'$ of positive $S$-dimension. We have to prove that for any
$S$-dense open subscheme $V$ of $X''$ then $f''_{V}$ is a weak
compression of $f'$. But, by \reftext{Lemma~\ref{lem:dominantcompressions}},
$f''_{V}$ is an $S$-defined everywhere compression of $f_{U}$ for some
$S$-dense open subscheme $U$ of $X$. Since $f$ is classifying then
$f_{U}$ is a weak compression of $f'$ and so, by \reftext{Lemma~\ref{lem:compositionofcompressions}} (i), we have that $f''_{V}$ is a
weak $S$-compression of $f'$.
\end{proof}

\begin{prop}
\label{prop:generictorsor}
The essential dimension of a group scheme $G$ over $S$ is equal to the
essential dimension of an $S$-classifying $G$-torsor, if it exists.
\end{prop}
\begin{proof}
Clearly the essential dimension of $G$ is greater than or equal to the
essential dimension of any $G$-torsor, in particular of an
$S$-classifying $G$-torsor. We will now prove that the essential
dimension of any $G$-torsor is at most the essential dimension of an
$S$-classifying $G$-torsor.

Let $f':Y'\to X'$ be a $G$-torsor and let $f:Y\to X$ be an
$S$-classifying $G$-torsor. Let $n$ be the essential dimension of
$f$. Moreover we can suppose that the $S$-dimension of $X'$ is strictly
positive otherwise there is nothing to prove. So let us consider a weak
$S$-compression $f'':Y''\to X''$ of $f$ such that $\dim_{S} X''=n$. By
\reftext{Lemma~\ref{lem:dominantcompressions}} there exists an $S$-dense open
subscheme $V$ of $X$ such that $f''$ is a defined everywhere weak
$S$-compression of $f_{V}$. Now $f$ is an $S$-classifying $G$-torsor so
$f_{V}$ is a weak $S$-compression of $f'$, then $f''$ is a weak
compression of $f'$ by \reftext{Lemma~\ref{lem:compositionofcompressions}} (i).
So the essential dimension of $f'$ is less than or equal to $n$.
\end{proof}

We now give a condition for the existence of an $S$-classifying torsor,
in the case $S$ is local.
\begin{prop}
\label{Proposition:linearaction}
Let us suppose $S$ is local. Let $G$ be a group scheme and let us
suppose that $G$ acts linearly on $\mathbb{A}^{n}_{S}$ and that there
exists an $S$-dense $G$-stable open subscheme $Y$ of $\mathbb{A}^{n}
_{S}$ such that we have an induced $G$-torsor $f:Y\to X$, with $X$ an
object of $\mathfrak{C}_{S}$. Then $f$ is an $S$-classifying $G$-torsor.
In particular the essential dimension of $G$ is finite and less than or
equal to $n-\dim_{S} G$.
\end{prop}

\begin{proof}
The last statement is clear.

Moreover we remark that since $Y$ is an open subscheme of an affine
space it is an object of $\mathfrak{C}_{S}$. The same is then true for
$X$ since $Y\to X$ is $S$-dominant (since faithfully flat). We have now
to prove that, for any $V$ open subscheme of $X$ faithfully flat over
$S$ and for any $G$-torsor $f':Y'\to X'$, with $X'$ an object of
$\mathfrak{C}_{S}$ of positive dimension, that $f_{V}$ is a weak
compression of $f'$. We can clearly suppose $V=X$.

We proceed exactly as in \cite[Theorem 4.1]{Me} where the result is
proved in the case of a field. We repeat the proof just to point out
where the hypothesis that $S$ is local is used.

Let us consider the $G\times_{S} G$-torsor $Y'\times_{S} Y\to X'
\times_{S} X$. If we quotient by the diagonal we have a $G$-torsor
$Y'\times Y\to Z$. Now we have that $Y'\times_{S} Y\to Y'$ is an open
subscheme of the trivial vector bundle $Y'\times_{S} \mathbb{A}^{n}
_{S}\to Y'$. Since $G$ acts linearly on $\mathbb{A}^{n}_{S}$ then we
have that $Y'\times_{S} \mathbb{A}^{n}_{S}\to Y'$ descends to a vector
bundle $W$ over $X'$ which contains $Z$ as an open subscheme. For any
point $x$ of $X'$ there exists an open subscheme $U'$ containing $x$
such that the vector bundle is trivial. Let us take $x$ in the preimage
of the closed point of $S$ under the morphism $\pi : X'\to S$. Since
$\pi (U')$ is open, $S$ is local and $\pi (U')$ contains the closed
point of $S$ then it is equal to $S$. So $U'$ intersects any fiber of
$X'$ over $S$. Since $X'$ has integral fibers we have that $U'$ is
$S$-dense. Then by the following Lemma we have a rational section~$s$,
defined over $U'$, of the vector bundle which factorizes through~$Z$.
Then we have finally, by \reftext{Lemma~\ref{lem:fiberproduct}}, a weak
$S$-compression given by
\begin{equation*}
\vcenter{\xymatrix{Y'\ar@{-->}^g[r]\ar^{f'}[d]& Y'\times_S Y \ar^{p_Y}[r]\ar[d]&
Y\ar^{f}[d]\\ X'\ar@{-->}^{s}[r]& Z \ar[r]& X}}.
\end{equation*}
So we are done using \reftext{Lemma~\ref{lem:compositionofcompressions}}(i).
\end{proof}

\begin{lemm}
Let us suppose $S$ local and $X\to S$ is a faithfully flat morphism
locally of finite presentation with integral fibers of positive
dimension. For any $S$-dense open subscheme $V$ of a linear affine space
$\mathbb{A}^{n}_{S}$ there exists an affine $S$-dense open subscheme
$U$ of $X$ with a morphism $U\to V$.
\end{lemm}

\begin{proof}
We can suppose that $X$ is affine taking an affine open subscheme with
non-empty fiber over the closed point of $S$. Since $S$ is local, as in
the proof of the Proposition, one proves that this open subscheme maps
onto $S$ and so, since fibers of $X\to S$ are integral, it is an
$S$-dense open subscheme of $X$. In particular it is faithfully flat
over~$S$. Moreover we can assume that $V$ is a principal open. So let
$X=\operatorname{Spec}(A)$ and $V=\operatorname{Spec}(R[T_{1},\dots ,T
_{n}]_{f})$ with $f\in R[T_{1},\dots ,T_{n}]$ but with at least one
invertible coefficient. Let $\mathfrak{m}$ be the maximal ideal of $R$.
Since $A/\mathfrak{m} A$ is infinite then there exist $a_{1}, \dots
,a_{n}\in A$ such that $h:=f(a_{1},\dots ,a_{n})$ is nonzero modulo
$\mathfrak{m}$. So, $T_{i}\mapsto a_{i}$ for $i=1,\dots , n$, gives a
morphism $\operatorname{Spec}(A_{h})\to V$ and $
\operatorname{Spec}(A_{h})$ is faithfully flat over $S$.
\end{proof}

The following result generalizes \cite[Remark 4.12]{BF}, which works
over a field.
\begin{coro}
\label{cor:classyfingtorsor}
Let us suppose $S$ local. If
\begin{enumerate}[(ii)]
\item[(i)] $G$ is a closed subgroup scheme of
    $\operatorname{GL}_{n,S}$  and
\item[(ii)] there exists an $S$-dense open subscheme $U$ of
    $\operatorname{GL} _{n,S}$ such that the schematic quotient
    $U/G$ exists and $U\to U/G$ is a $G$-torsor,
\end{enumerate}
then $U\to U/G$ is an $S$-classifying $G$-torsor. In particular the
essential dimension of $G$ is finite and less than or equal to
$n^{2}-\dim_{S} G$.
\end{coro}

\begin{proof}
Let us consider $\operatorname{GL}_{n,S}$ contained in $\mathbb{A}
^{n^{2}}_{S}$. If we view $\mathbb{A}^{n^{2}}_{S}$ as the scheme which
represents the functor of square matrices of order $n$ then
$\operatorname{GL}_{n_{S}}$ acts on it by multiplication on the right.
In fact it acts freely, i.e. it acts freely on $T$-points, with $T$ any
$S$-scheme. Now by condition $(\mathit{ii})$, $U\to U/G$ is a $G$-torsor. It is
easy to verify that $U/G$ is an object of $\mathfrak{C}_{S}$. So by the
above proposition it is a classifying $G$-torsor.
\end{proof}

\begin{rema}
\label{rem:existencestandardtorsor}
We recall that any affine flat group scheme of finite type over an
affine regular noetherian scheme $T$ of dimension ${\le} 2$ is isomorphic
to a closed subgroup scheme of $\operatorname{GL}_{n,T}$ for some
$n$ (see \cite[Expos\'{e} VI, Proposition 13.2]{SGA3_I}). So any group
scheme (recall our conventions at the beginning of the section) over a regular noetherian local scheme of dimension ${\le} 2$
satisfies the first condition. Moreover condition $(i)$ is always
satisfied for finite  group schemes over $S$. In fact the usual
proof which works for fields (see for instance \cite[\S 3.4]{Wat}) works
also for finite group schemes of finite presentation over $S$.
The main point is to find a finite free $\mathcal{O}_{S}$-module $M$
where $G$ acts faithfully, i.e. the morphism of sheaves $G\to
\operatorname{GL}_{n}(M)$ is injective. But in fact firstly one takes
$\mathcal{O}_{G}$. Since $G$ is flat and of finite presentation over
$S$ then $\mathcal{O}_{G}$ is projective as $\mathcal{O}_{S}$-module
\cite[\S 1.5, Corollaire]{Bou_A_X}. So it is a direct factor of a
\textit{finite} (since it is finitely generated) free $\mathcal{O}
_{S}$-module $M$. Then $G$ acts faithfully on $M$.

We remark that, supposing the first condition is verified, the second
one is satisfied if $S$ is of dimension ${\le} 1$. This follows from
\cite[Theorem 4.C and Theorem 3.1.1]{An}. In fact one can take
$U=\operatorname{GL}_{n,S}$. In general, if $(i)$ is satisfied, there
exists always an open subscheme $U$ with that property, by
\cite[Expos\'{e}~V, Th\'{e}or\`{e}me~10.4.2]{SGA3_I}, since the action
of $G$ over $\operatorname{GL}_{n,S}$ is free. But we do not know if we
can take it $S$-dense in general.
\end{rema}

\begin{defi}
We will call a \textit{standard} torsor any torsor which satisfies
conditions of Corollary (even if the base is not local).
\end{defi}

Since the hypotheses of \reftext{Corollary~\ref{cor:classyfingtorsor}} are stable
by base change then if $G$ has a standard torsor over $S$ then it has
a standard torsor, by base change, over $T$ for any morphism
$T\to S$.

\begin{coro}
\label{coro:edandbasechange}
Let us suppose $S$ local and let us suppose there exists a standard
$G$-torsor over $S$. Let $T\to S$ be a morphism of schemes with $T$
local, integral and noetherian. Then $\operatorname{ed}_{T} G_{T}
\le \operatorname{ed}_{S} G$.
\end{coro}
\begin{rema}
If $S$ and $T$ are spectra of fields this corollary is exactly
\cite[Proposition~1.5]{BF}, even if the proof is different. The above
corollary can be applied, for instance, with $T$ a point of $S$.
\end{rema}

\begin{proof}
Since the pull-back of a standard torsor is a standard torsor then the
result follows by \reftext{Lemma~\ref{lem:basechange}} and
\reftext{Proposition~\ref{prop:generictorsor}}.
\end{proof}

We obtain the following result.
\begin{coro}
\label{coro:edsufibraspeciale}
Let $k$ be a field, let $G$ be a $k$-group scheme and let us suppose
that $S$ is a local $k$-scheme. For any point $x$ of $S$ with residue
field $k(x)$
\begin{equation*}
\operatorname{ed}_{k} G\ge \operatorname{ed}_{S} G_{S}\ge
\operatorname{ed}_{k(x)} G_{k(x)}.
\end{equation*}
In particular if the closed point is $k$-rational then
\begin{equation*}
\operatorname{ed}_{k} G= \operatorname{ed}_{S} G_{S}.
\end{equation*}
\end{coro}
\begin{proof}
This follows by the above corollary, since $G_{S}\times_{S}
\operatorname{Spec}(k(x))\simeq G_{k(x)}$. For the last part, since
$x$ is $k$-rational, then $k(x)=k$ and $G_{k(x)}= G$.
\end{proof}

We finally prove that if $S$ is a field we recover the usual definition
of essential dimension.

\begin{prop}
\label{Prop:edoverafield}
If $k$ is a field then \reftext{Definition~\ref{def:essentialdimension}} is
equivalent to the usual definition of essential dimension of a group
scheme (\reftext{Definition~\ref{def:edusuale}}).
\end{prop}
\begin{proof}
In \cite[Corollary 6.16 and Lemma 6.11]{BF} it is proved that the
essential dimension of $G$ (as in \reftext{Definition~\ref{def:edusuale}}) is
the dimension of $X'$ where $Y'\to X'$ is a $k$-compression of
$Y\to X$ and $Y\to X$ is a classifying $G$-torsor. By
\cite[Remark 4.12]{BF} we can suppose it is standard. We remark that in
\cite{BF} (see also \reftext{Definition~\ref{def:classicalessentialdimension}}) it is not assumed that $X'$
is an object of $\mathfrak{C}_{k}$. So, a priori, it could be not
geometrically integral. But since $Y$ is an open subscheme of an affine
space it is geometrically integral. The same is then true for $X$ since
$Y\to X$ is schematically dominant (since faithfully flat). Finally
since also $X\dashrightarrow X'$ is schematically dominant, being
dominant between reduced schemes, then also $X'$ is geometrically
integral. We also remark that in \reftext{Definition~\ref{def:essentialdimension}} we admit $k$-weak compressions. But this
is not at all a problem, since over a field, as explained after
\reftext{Definition~\ref{def:essentialdimension}}, we can suppose to work with
$k$-compressions.
\end{proof}

We have this unsurprising result.

\begin{lemm}
\label{lemm:edsottogruppi}
Let us suppose $S$ is local. Let $H$ be a closed (faithfully flat) $S$-subgroup
scheme of a group scheme $G$ over $S$ and suppose that $G$ has a
standard torsor. Then
\begin{equation*}
\operatorname{ed}_{S} H+\dim_{S} H\le \operatorname{ed}_{S} G +\dim
_{S} G.
\end{equation*}
\end{lemm}
\begin{proof}
The proof follows \cite[Theorem 6.19]{BF}, which gives the result over
a field.

Take a standard $G$-torsor $f:U\to X$. By \reftext{Proposition~\ref{prop:generictorsor}} and \reftext{Corollary~\ref{cor:classyfingtorsor}} the
essential dimension of $G$ over $S$ is equal to the essential dimension
of this torsor. We remark that $g:U\to U/H$ is a standard classifying
$H$-torsor. Now let $f':U'\to X'$ be an $S$-weak compression of $f$ such
that $\operatorname{ed}_{S} G=\dim_{S} X'$. Then $g':U'\to U'/H$ is an
$S$-weak compression of $g$. Therefore
\begin{align*}
\operatorname{ed}_{S} H\le \dim_{S}(U'/H)
&=\dim_{S} U'-\dim_{S} H
\\
&=\dim_{S} G+\dim_{S} X'-\dim_{S} H
\\
&=\operatorname{ed}_{S} G+\dim_{S} G-\dim_{S} H
\end{align*}
and we are done.
\end{proof}

\section{Invariance of essential dimension by base change}\label{sec4}

In this section we see when the essential dimension over a field remains
invariant if we change base field. In the following $S$ is, as above,
an integral locally noetherian scheme. And see the beginning of section
\S \ref{sectio:compressions} for the assumptions on group schemes. We first prove the following
result.
\begin{theo}
\label{thm:semicontinuityed}
Let $G$ be a group scheme over $S$. Then there exists a non-empty open
subscheme $U$ of $S$ such that for any $s\in U$, with residue field
$k(s)$,
\begin{equation*}
\operatorname{ed}_{k(s)}(G_{k(s)})\le \operatorname{ed}_{k(S)}(G_{k(S)})=
\operatorname{ed}_{\mathcal{O}_{S,s}}G_{\mathcal{O}_{S,s}}.
\end{equation*}
\end{theo}
\begin{proof}
By \reftext{Lemma~\ref{lem:basechange}} we have just to prove that there exists
a non-empty open subscheme $U$ of $S$ such that, for any $s\in U$,
$\operatorname{ed}_{\mathcal{O}_{S,s}}(G_{\mathcal{O}_{S,s}})\le
\operatorname{ed}_{k(S)}(G_{k(S)})$. We know that $G_{k(S)}$ has a
closed immersion $\iota $ in some $\operatorname{GL}_{n,k(S)}$. We
observe that $\operatorname{Spec}(k(S))=\lim\limits_{\leftarrow }U$
where $U$ ranges through affine open subschemes of $S$. By
\cite[Expos\'{e}~VI Proposition~10.16]{SGA3_I} and
\cite[Th\'{e}or\`{e}me 8.8.10]{EGAIV3} there exists an affine open
subscheme $V$ of $S$ such that $\iota $ extends (uniquely up to restrict
the open subscheme) to a closed immersion $G_{V}\to \operatorname{GL}_{n,V}$. The
closed immersion $\iota $ gives a standard torsor $f:Y\to X$ over
$\operatorname{Spec}(k(S))$. By \cite[Th\'{e}or\`{e}me 8.8.2]{EGAIV3}
up to restrict $V$ we can suppose that $f$ extends to a morphism
$Y'\to X'$ over $V$. Again by \cite[Th\'{e}or\`{e}me 8.8.2]{EGAIV3} up
to restrict $V$ we can suppose that the $X$-action of $G_{k(S)}$ over
$Y$ extends to a $X'_{V}$-action of $G_{V}$ over $X'_{V}$. Finally, by
\cite[Expos\'{e} VI Proposition 10.16]{SGA3_I}, up to restrict again
$V$, we can suppose that $\tilde{f}:Y'_{V}\to Z'_{V}$ is a
$G_{V}$-torsor. We remark that it is a standard $G_{V}$-torsor
associated to the above extension of $\iota $.

Let $g:W\to Z$ be a $k(S)$-weak compression of $f_{k(S)}$ such that
$\dim_{k(S)}Z=\operatorname{ed}_{k(S)} G_{k(S)}$, which is possible
since $f_{k(S)}$ is classifying. Reasoning as above this $k(S)$-weak
compression extends to a weak $U$-compression of $\tilde{f}_{U}$, where
$U$ is an affine open subscheme contained in $V$. Let $s$ be a point of
$U$. If we pull-back to $\operatorname{Spec} (\mathcal{O}_{S,s})$ we
have a weak $\operatorname{Spec}(\mathcal{O}_{S,s})$-compression of the
standard $G_{\mathcal{O}_{X,x}}$-torsor $f_{\operatorname{Spec}(
\mathcal{O}_{S,s})}$. So we have that $\operatorname{ed}_{\mathcal{O}
_{S,s} }G_{\mathcal{O}_{S,s}}\le \operatorname{ed}_{k(S)}G_{k(S)}$ as
wanted.
\end{proof}

\begin{exam}
If $G$ is a group scheme over $\mathbb{Z}$ we have that
\begin{equation*}
\operatorname{ed}_{\mathbb{Q}} G_{\mathbb{Q}}\ge \operatorname{ed}
_{\mathbb{F}_{p}} G_{\mathbb{F}_{p}}
\end{equation*}
for all except possibly finitely many primes $p$.
\end{exam}

Now we state some corollaries. We recall that a group scheme over a
field $L$ is called almost special (as defined in \cite{TV}) if
$\operatorname{ed}_{L}G=\dim_{L} \Lie G -\dim_{L} G$. In
\cite[Theorem~1.2]{TV} has been proved that this is the minimal value
which can be obtained. For instance trigonalizable group schemes of
height ${\le} 1$ are almost special (see \cite[Corollary 4.5]{TV}). We recall that a
group scheme over a field of characteristic $p$ is of height ${\le} n$
if it is killed by $\operatorname{F}^{n}$ where $\operatorname{F}$ is
the Frobenius. We do not know examples of trigonalizable group schemes
not almost special even if we think that there are a lot of them. By
the previous Theorem it follows the following corollary.
\begin{coro}
\label{coro:quasispecialisopen}
Let $G$ be a group scheme over $S$ such that $G_{k(S)}$ is almost
special. Then $\operatorname{ed}_{k(S)} G_{k(S)}\le \operatorname{ed}
_{k(s)} G_{s}$ for any $s\in S$ and there exists a non-empty open
subscheme $U$ of $S$ such that, for any $s\in U$, $G_{s}$ is almost
special and $\operatorname{ed}_{k(s)} G_{s}=\operatorname{ed}_{k(S)} G
_{k(S)}$.
\end{coro}
\begin{proof}
We have $\operatorname{ed}_{k(S)}{G_{k(S)}}=\dim_{k(S)} \Lie G_{k(S)}
-\dim G_{k(S)}$. By \cite[Theorem~1.2]{TV}, $\operatorname{ed}_{k(s)}
{G_{s}}\ge \dim_{k(s)} \Lie G_{s} -\dim_{k(s)} G_{s}$ for any
$s\in S$. Moreover, the dimension of the tangent space is an upper
semicontinuous function. Then $\dim_{k(s)} \Lie G_{s}\ge \dim_{k(S)}\Lie G
_{k(S)}$ for any $s\in S$ and there exists a non-empty open subscheme
$V$ of $S$ such that $\dim_{k(s)} \Lie G_{s}= \dim_{k(S)}\Lie G_{k(S)}$.
So for any $s\in S$, $\operatorname{ed}_{k(s)} G_{s}\ge
\operatorname{ed}_{k(S)} G_{k(S)}$ and $\operatorname{ed}_{k(S)} G=
\dim_{k(s)} \Lie G_{s} -\dim G_{s}$ for any $s\in V$. Finally, by the
\reftext{Theorem~\ref{thm:semicontinuityed}} there exists a  non-empty open
subscheme $U\subseteq V$ of $S$ such that $\operatorname{ed}_{k(s)}G
_{s}\le \operatorname{ed}_{k(S)} G_{k(S)}$. Then, for any $s\in U$,
$G_{s}$ is almost special and $\operatorname{ed}_{k(s)} G_{s}=
\operatorname{ed}_{k(S)} G_{k(S)}$.
\end{proof}

We also have this corollary.

\begin{coro}
\label{Thm:homotopytheorem}
Let $k$ be a field, let $X$ be an integral scheme of finite type over
$k$ with fraction field $k(X)$ and let $G$ be a group scheme over
$k$. There exists a non-empty open subscheme $U$ of $X$ such that for
any $x\in U$, with residue field $k(x)$,
\begin{equation*}
\operatorname{ed}_{k(x)}(G_{k(x)})\le \operatorname{ed}_{k(X)}(G_{k(X)})=
\operatorname{ed}_{\mathcal{O}_{X,x}}G_{\mathcal{O}_{X,x}}.
\end{equation*}
In particular if the set of $k$-rational points of $X$ is Zariski dense
then
\begin{equation*}
\operatorname{ed}_{k}(G)=\operatorname{ed}_{k(X)}(G_{k(X)}).
\end{equation*}
\end{coro}
\begin{proof}
We have just to apply \reftext{Theorem~\ref{thm:semicontinuityed}} to
$G_{X}$ over $X$ and we are done. The last statement is clear using
\reftext{Corollary~\ref{coro:edsufibraspeciale}}.
\end{proof}

\begin{rema}
\begin{enumerate}[(ii)]
\item[(i)] The requirement about rational points is necessary. In fact, otherwise,
one can take a finite extension of fields $k\subseteq k'$ and a group
$G$ such that $\operatorname{ed}_{k} G>\operatorname{ed}_{k'}G$. However
we do not have a counterexample with $X$ geometrically integral over
$k$ of positive dimension. See also discussion before \reftext{Proposition~\ref{Proposition:finitefield}}.
\item[(ii)] If the set of $k$-rational points are dense then, in
    particular, there exists a Zariski-dense set of closed points
    $x$, given by $k$-rational points, such that
    $\operatorname{ed}_{k(x)}G=\operatorname{ed}_{k(X)}G$. One
    could ask if this works in general, i.e. does there exist a
    Zariski-dense set $C$ of closed points such that for any $x\in
    C$ we have
    $\operatorname{ed}_{k(x)}G_{k(x)}=\operatorname{ed}_{k(X)}G_{k(X)}$?
    The answer is no. In fact take a geometrically integral variety
    $X$ over $\mathbb{R}$ without $\mathbb{R}$-rational points,
    e.g. the affine curve over $\mathbb{R}$ defined by
    $x^{2}+y^{2}=-1$. Then, by \cite[Theorem~7.6]{BF}, we have that
    $\operatorname{ed}_{\mathbb{R}(X)} \mathbb{Z}/4
    \mathbb{Z}=\operatorname{ed}_{\mathbb{R}} \mathbb{Z}/4=2$ since
    $\mathbb{R}$ and $\mathbb{R}(X)$ have no square roots of $-1$.
    But any closed point of $X$ has $\mathbb{C}$ as residue field.
    And we have that $\operatorname{ed}_{\mathbb{C}}
    \mathbb{Z}/4\mathbb{Z} = \operatorname{ed}_{\mathbb{C}} \mu_{4}
    =1$.
\end{enumerate}
\end{rema}

We recall that a field $k$ is \textit{pseudo algebraically closed} if
any (non-empty) integral geometrically irreducible scheme $X$ of finite
type over $k$ has at least $k$-rational point. It is clear that this is
equivalent to say that for any such an $X$ the set of $k$-rational point
of $X$ is dense.
\begin{coro}
If $k$ is a pseudo algebraically closed field and $X$ an integral
geometrically irreducible scheme of finite type over $k$ with fraction
field $k(X)$ then
\begin{equation*}
\operatorname{ed}_{k}(G)=\operatorname{ed}_{k(X)}(G_{k(X)}),
\end{equation*}
for any group scheme $G$ over $k$.
\end{coro}
\begin{proof}
This follows just by definition of pseudo algebraically closed and by
the last part of the \reftext{Corollary~\ref{Thm:homotopytheorem}}.
\end{proof}

\begin{coro}
\label{Cor:dimensioneessenzialeminima}
Let $G$ be a group scheme over a field $k$. Then
\begin{equation*}
\operatorname{ed}_{\bar{k}} G_{\bar{k}} \le \operatorname{ed}_{K} G
_{K}
\end{equation*}
for any extension $K$ of $k$, where $\bar{k}$ is an algebraic closure
of $k$.
\end{coro}

\begin{proof}
We first restrict to \textit{finitely generated} extensions $K$ of
$k$. Let $X$ be an integral scheme of finite type over $k$ with fraction
field $K$. Then by \reftext{Corollary~\ref{Thm:homotopytheorem}} there exists
a closed point $x$ of $X$, since closed points are dense in $X$, such
that $\operatorname{ed}_{k(x)}G_{k(x)}\le \operatorname{ed}_{K}G_{K}$.
But since $x$ is closed then $k(x)$ is a finite extension of $k$, so
\begin{equation*}
\operatorname{ed}_{\bar{k}} G_{\bar{k}}\le \operatorname{ed}_{k(x)}G
_{k(x)}\le \operatorname{ed}_{K} G_{K}
\end{equation*}
Now the general case follows from the following lemma.
\end{proof}

\begin{lemm}
\label{lemm:possiamosupporrefingen}
Let $G$ be a group scheme over $k$ and $K$ an extension of $k$. Then
there exists a finitely generated extension $L$ of $k$ such that
\begin{equation*}
\operatorname{ed}_{K} G_{K}=\operatorname{ed}_{L}G_{L}.
\end{equation*}
\end{lemm}
\begin{proof}
Take a standard classifying $G$-torsor $f:Y\to X$. Then we have that
$f_{K}:Y_{K}\to X_{K}$ is a standard classifying $G_{L}$-torsor. Let
$Y'\to X'$ a compression of $f_{K}$ such that $\dim_{K} X'=
\operatorname{ed}_{K}G_{K}$. The $G_{K}$-torsor $Y'\to X'$ and the
compression are in fact defined over a finitely generated extension
$L$ of $k$. So we obtain a compression of $f_{L}: Y_{L}\to X_{L}$. This
means that $\operatorname{ed}_{L} G_{L}\le \dim_{K} X'=
\operatorname{ed}_{K} G_{K}$. Since we always have the opposite
inequality we are done.
\end{proof}

\begin{coro}
Let $G$ be a group scheme over a field $k$. Then the following are
equivalent:
\begin{enumerate}[(ii)]
\item[(i)]
\begin{equation*}
\operatorname{ed}_{k} G=\operatorname{ed}_{k'} G_{k'},
\end{equation*}
for any finite extension $k'$ of $k$.
\item[(ii)]
\begin{equation*}
\operatorname{ed}_{k}G=\operatorname{ed}_{K}G_{K},
\end{equation*}
for any extension $K$ of $k$.
\item[(iii)] $\operatorname{ed}_{k} G=\operatorname{ed}_{\bar{k}}
    G_{\bar{k}}$.
\end{enumerate}

\end{coro}
\begin{proof}
It is clear that $(\mathit{ii})\Rightarrow (\mathit{iii})$. Now $(\mathit{iii}) \Rightarrow (i)$
since any finite extension of $k$ is included in $\bar{k}$.

We finally prove $(i)\Rightarrow (ii)$. The proof is similar to that one
of \reftext{Corollary~\ref{Cor:dimensioneessenzialeminima}}. By \reftext{Lemma~\ref{lemm:possiamosupporrefingen}} we can suppose $K$ finitely
generated over $k$. Then $K$ is the fraction field of an integral
variety $X$. Then by \reftext{Corollary~\ref{Thm:homotopytheorem}} there exists
a closed point $x$ of $X$ such that $\operatorname{ed}_{k(x)}G_{k(x)}
\le \operatorname{ed}_{K}G_{K}\le \operatorname{ed}_{k}G$. But since
$x$ is closed then $k(x)$ is a finite extension of~$k$, so we are done.
\end{proof}

\begin{exam}
If $k$ is algebraically closed and $G$ is a group scheme over $k$ then
$\operatorname{ed}_{k} G=\operatorname{ed}_{K} G_{K}$ for any extension
$K$ of $k$. This is also proven in \cite[Proposition 2.14]{BRV}.
\end{exam}

As explained in the introduction it would be interesting to answer to
the following question.

\textit{(Q) if $X$ is an integral scheme of finite type over $k$ and
$ \operatorname{ed}_{k}(G)=\operatorname{ed}_{k(X)}(G_{k(X)}) $ for any
group scheme $G$, is the set of $k$-rational points of $X$
Zariski-dense?}

If the answer was positive (at least over number fields), and we are not
so optimistic, the Lang conjecture, i.e. the set of rational points of
varieties of general type over a number field is not Zariski-dense,
could be rephrased in terms of essential dimension, giving, possibly,
a new point of view. Namely the Lang conjecture would be rewritten as:

\textit{(L) if $X$ is a variety of general type over a number field
$k$ and with fraction field $K$ then there exists a group scheme $G$
over $k$ such that $\operatorname{ed}_{k}{G}\neq \operatorname{ed}
_{K}G_{K}$.}

Nevertheless we remark that the above statement (L) always implies,
using \reftext{Corollary~\ref{Thm:homotopytheorem}}, Lang Conjecture. Only the
converse is linked to the question (Q).

The positive answer to the question $(Q)$ for varieties over
$\mathbb{F}_{p}$ (generalizable to finite fields) is given essentially
in \cite[Prop 3.6 and Lemma 4.5]{NDT}. We remark that the set of
rational points of any positive dimensional variety over a finite field
is not Zariski dense. The following result is just a refinement of
\cite[Prop 3.6]{NDT}. The idea of the proof is the same, we just
slightly improved it to include, for instance, the essential dimension
of $(\mathbb{Z}/p\mathbb{Z})^{2}$ over $\mathbb{F}_{p}$.
\begin{prop}
\label{Proposition:finitefield}
Let $K$ be a field of characteristic $p$. Then
\begin{equation*}
\operatorname{ed}_{K}(\mathbb{Z}/p\mathbb{Z})^{r}= \left\{
\begin{array}{l@{\quad }l}
2 & \mbox{if $K$ is finite of order less than $p^{r}$}
\\
1 & \mbox{otherwise.}
\end{array}
\right.
\end{equation*}
In particular if $K$ is the function field of a positive dimensional
variety over $\mathbb{F}_{p}$ and $q=p^{r}$ then $\operatorname{ed}
_{\mathbb{F}_{q}}(\mathbb{Z}/p\mathbb{Z})^{s} > \operatorname{ed}_{K}(
\mathbb{Z}/p\mathbb{Z})^{s}$ if $s>r$.

\end{prop}

\begin{proof}
Suppose that $K$ is finite of order greater than or equal to
$p^{r}$. Then it contains an $\mathbb{F}_{p}$-vector space of dimension
$p^{r}$ and so we can embed $(\mathbb{Z}/p\mathbb{Z})^{r}$ in
$\mathbb{G}_{a,K}$, which gives $\operatorname{ed}_{K}((\mathbb{Z}/p
\mathbb{Z})^{r})\le \dim_{K}\mathbb{G}_{a,K}+\operatorname{ed}_{K}
\mathbb{G}_{a,K}=1$ (see \cite[Proposition 5]{Le1} for a more general
statement).

Now we suppose we are in the other situation. By \cite[Lemma 7.2]{BF},
we have that if $\operatorname{ed}_{K}(\mathbb{Z}/p\mathbb{Z})^{r}=1$
then $(\mathbb{Z}/p\mathbb{Z})^{r}$ is isomorphic to a subgroup of
$\operatorname{PGL}_{2}(K)$. Let $q$ be the order of $K$. It is easy to
see that $\operatorname{PGL}_{2}(K)$ has order $(q+1)(q^{2}-q)=q(q
^{2}-1)$. Therefore it has no $p$-subgroups of order greater than
$q$. Therefore $\operatorname{ed}_{K}(\mathbb{Z}/p\mathbb{Z})^{r}>1$.
On the other hand by \cite[Lemma 3.5]{NDT} we have that the essential
dimension of an elementary $p$-abelian group in characteristic $p$
should be less than or equal to $2$. So we are done.

The last part is clear.
\end{proof}

\section{Essential dimension over a discrete valuation ring}

In this section let $R$ be a discrete valuation ring with residue field
$k$ of characteristic $p>0$ and fraction field $K$. We set $S=
\operatorname{Spec}(R)$.
We recall that for group scheme over $S$ we will mean an affine faithfully flat group scheme of finite presentation over $S$. So in particular any group scheme will be automatically \textit{flat} over the base. If $G$ is a group scheme over $K$, a
\textit{model} of $G$ is a group scheme $\mathcal{G}$ over $S$ with
an isomorphism $\mathcal{G}\times_{S} K\to G$. If $G$ finite we require
$\mathcal{G}$ finite over $S$, if not differently specified. We observe
that any finite (flat) group scheme over $S$ has a standard $G$-torsor
(see \reftext{Remark~\ref{rem:existencestandardtorsor}}).

We know that $\mathbb{G}_{a}$ and $\mathbb{G}_{m}$ have essential
dimension zero over any field. We have the following result.

\begin{prop}
\label{prop:edmodelliGa}
A model of $\mathbb{G}_{m,K}$ or $\mathbb{G}_{a,K}$ has essential
dimension zero if and only if it is smooth with connected fibers.
\end{prop}
\begin{proof}
We first prove the \textit{if part}. It is known by
\cite[Theorem 2.2]{WW} that $\mathbb{G}_{a,S}$ is the unique smooth
model with connected fibers of $\mathbb{G}_{a,K}$ and by \cite[Theorem
2.5]{WW} that any smooth model with connected fibers of $\mathbb{G}
_{m,K}$ is isomorphic to
\begin{equation*}
\mathcal{G}^{\lambda }=\operatorname{Spec}(R[T,\frac{1}{1+\lambda T}])
\end{equation*}
for some $\lambda \in R\setminus \{0\}$, where the law group is the
unique one such that the morphism
\begin{equation*}
\mathcal{G}^{\lambda }\longrightarrow \mathbb{G}_{m,S}
\end{equation*}
given by $T\mapsto 1+\lambda T$ is a morphism of group schemes. This is
clearly an isomorphism on the generic fiber. These groups depend only
on the valuation of $\lambda $. We remark that if $\lambda =0$ then
$\mathcal{G}^{\lambda }\simeq \mathbb{G}_{a,S}$.

Now let $Y\to X$ be a $\mathbb{G}_{a,S}$-torsor with $X$ object of
$\mathfrak{C}_{S}$. It is well known that this torsor is locally trivial
in the Zariski topology. Here the important point in fact is that it is
trivial over an $S$-dense, using the argument already used in the proof
of \reftext{Proposition~\ref{Proposition:linearaction}}. Therefore the essential
dimension is zero.

Now, also a $\mathcal{G}^{\lambda }$-torsor, with $\lambda \neq 0$, is
locally trivial in the Zariski topology. This has been proved in
\cite[Proposition 2.3.1]{To}. So, as above, it is trivial over an
$S$-dense and therefore the essential dimension is zero.

Just to be precise in \cite[Proposition 2.3.1]{To} it has been proved
that the first group of cohomology of $\mathcal{G}^{\lambda }$ in the
\textit{small} fppf site is the same of the first group of cohomology
in the \textit{small} Zariski site over a scheme~$X$. Small fppf site
means that the category you are considering is that one
of \textit{flat of locally finite presentation} schemes over~$X$.
Similarly for the small Zariski site you are considering just Zariski
open sets of~$X$. But by \cite[III Proposition 3.1]{Mi} \footnote{In
fact in Milne's book all categories are supposed to have fiber
products. This is not the case for the category of flat schemes over
$X$. But however this is not necessary since one only needs that fiber
products with schemes involved in the coverings exist. So, for
instance, \cite[III Lemma 1.19]{Mi}, which is used in \cite[III
Proposition 3.1]{Mi} without be mentioned, holds also in the small fppf
site, and the proof is the same.} the cohomology is the same if you
take the small site or the big site (i.e. you consider the category of
all schemes locally of finite presentation over $X$)

The \textit{only if part} follows by the following Lemma.
\end{proof}

\begin{lemm}
Any group scheme over a noetherian integral scheme $T$ has
essential dimension greater than zero if one of its fibers is non-smooth
or not connected.
\end{lemm}
\begin{proof}
By \reftext{Corollary~\ref{coro:edandbasechange}} we have that the essential
dimension over $T$ is greater than or equal to the essential dimension
over any fiber. So we can conclude by \cite[Theorem~1.2]{TV}, for the
non-smooth case. Moreover in the proof of \cite[Proposition~4.3]{TV},
has been proved that a group scheme over a field with essential
dimension zero (so necessarily smooth) is connected. So we are done.
\end{proof}

\begin{lemm}
\label{lemm:modelssubgroupsGa}
Let us suppose that $K$ has characteristic $p>0$. Any model (not
necessarily finite) of a finite closed subgroup scheme of $\mathbb{G}
_{a,K}$ of order $p$ is isomorphic as $S$-group scheme to a closed
subgroup scheme of $\mathbb{G}_{a,R}$.
\end{lemm}
\begin{proof}
If models are finite this follows from more general statements about
classification of group schemes, like \cite{OT} (if $S$ is complete)
or \cite[Proposition 2.2]{DeJ}. \reftext{Lemma~\ref{lemm:modelssubgroupsGa}}
works however also for quasi-finite models. We give here a direct proof.
Let $G$ be a model of a finite subgroup scheme of $\mathbb{G}_{a,K}$.
Then there exists $x\in K[G]$ such that $\Delta (x)=x\otimes 1+1
\otimes x$, where $\Delta $ is the comultiplication. Then there exists
$a\in R$ such that $ax\in R[G]$. Since $\Delta (ax)=ax\otimes 1+1
\otimes ax$ the element $ax$ gives a morphism $G\to \mathbb{G}_{a,R}$.
Take the image of this morphism on the generic fiber and consider the
schematic closure $G'$, which is flat over $R$. So the induced morphism
$G\to G'$ is a model map, i.e. it is an isomorphism on the generic
fiber. Therefore, by \cite[Theorem 1.4]{WW}, it is a composition of
N\'{e}ron blow-ups (also called dilatations). Since $G$ has order
$p$, any blow-up is done over the trivial subgroup of the special fiber.
Now, since $G'\subseteq \mathbb{G}_{a,R}$, the first blow-up is
contained, by \cite[Proposition~2, \S 3.2]{BLR}, in the blow up of
$\mathbb{G}_{a,R}$ in the trivial subgroup scheme of the special fiber.
But, by the proof of \cite[Theorem 2.2]{WW}, this is isomorphic to
$\mathbb{G}_{a,R}$. Now continuing this process we have that $G$ is
isomorphic to a subgroup scheme of $\mathbb{G}_{a,R}$.
\end{proof}

In the proof of \reftext{Proposition~\ref{prop:edmodelliGa}} we recalled smooth
models of $\mathbb{G}_{m,K}$ with connected fibers, the so called
$\mathcal{G}^{\lambda }$, with $\lambda \in R\setminus \{0\}$. We
consider the isogeny $\mathcal{G}^{\lambda }\to \mathcal{G}^{\lambda
^{p}}$, with $v(p)\ge (p-1)v(\lambda )$ if $R$ has characteristic zero,
given by $T \mapsto ((1+\lambda T)^{p}-1)/p$. We note the kernel
$G_{\lambda ,1}$. It is a model of $\mu_{p,K}$, with natural isomorphism
on the generic fiber. In fact all models of $\mu_{p,K}$ are of this type
(see considerations just after \cite[Theorem 2.5]{WW}).

\begin{defi}
\label{df:filtered_gp_schemes}
For $i\in \{1,\dots ,n\}$ let $\lambda_{i}\in R\setminus \{0\}$.
\begin{enumerate}[(2)]
\item[(1)] A \emph{filtered $S$-group scheme of type
    $(\lambda_{1},\dots ,\lambda _{n})$} is a tuple
    $\mathcal{E}=(\mathcal{E}_{1},\dots ,\mathcal{E} _{n})$ of
    (affine) smooth commutative $S$-group schemes such that there
    exist exact sequences, with $1\le i \le n$ and
    $\mathcal{E}_{0}=0$:
\begin{equation*}
0\longrightarrow \mathcal{G}^{\lambda_{i}}\longrightarrow \mathcal{E}
_{i}\longrightarrow \mathcal{E}_{i-1}\longrightarrow 0.
\end{equation*}

\item[(2)] A \emph{Kummer group scheme of type $(\lambda_{1},\dots
    ,\lambda_{n})$} is a tuple $G=(G_{1},\dots ,G_{n})$ of finite
    (flat) commutative $S$-group schemes such that there exist a
    filtered $S$-group scheme $ \mathcal{E}=(\mathcal{E}_{1},\dots
    ,\mathcal{E}_{n})$ and commutative diagrams with exact rows,
    with $1\le i \le n$ and $G_{0}=0$:
\begin{equation*}
\begin{tikzpicture}[>=myto,scale=1.2,text height=1.5ex, text depth=0.5ex]
\node (01) at (0,-1) {$0$};
\node (A) at (1,-1) {$G_{\lambda_i,1}$};
\node (B) at (2.5,-1) {$G_i$};
\node (C) at (4,-1) {$G_{i-1}$};
\node (02) at (5,-1) {$0$};
\node (03) at (0,-2.2) {$0$};
\node (D) at (1,-2.2) {$\overset{\ }{\mathcal {G}^{\lambda_i}}$};
\node (E) at (2.5,-2.2) {$\mathcal {E}_i$};
\node (F) at (4,-2.2) {$\mathcal {E}_{i-1}$};
\node (04) at (5,-2.2) {$0$};
\draw[->,font=\scriptsize]
(01) edge (A)
(A) edge (B)
(B) edge (C)
(C) edge (02)
(03) edge (D)
(D) edge (E)
(E) edge (F)
(F) edge (04);
\path[right hook->] (A) edge (D);
\path[right hook->] (B) edge (E);
\path[right hook->](C) edge (F);
\end{tikzpicture}
\end{equation*}
\end{enumerate}
\end{defi}

We remark that if $G$ is a Kummer group scheme contained in a filtered
groups scheme $\mathcal{E}$ then $\mathcal{E}/G$ is filtered. Sometimes,
by abuse of notation, we will say that $\mathcal{E}_{n}$ (resp.
$G_{n}$), is a filtered $S$-group scheme (resp. Kummer group scheme).
We also stress the fact that, even if the generic fiber is
diagonalizable, almost always the special fiber is unipotent. 

These group schemes have been introduced and classified, at least in the
case the generic fiber is cyclic, in \cite{MRT1}. Conjecturally they
represent all models of finite diagonalizable $p$-group schemes. We have
the following result.
\begin{lemm}
\label{lemm:edKummer}
A filtered group scheme has essential dimension zero and a Kummer group
scheme of order $p^{n}$ has essential dimension less than or equal to
$n$.
\end{lemm}
\begin{proof}
The first statement follows by the fact that $\mathcal{G}^{\lambda
}$-torsors are locally Zariski trivial, as recalled in the proof of
\reftext{Proposition~\ref{prop:edmodelliGa}}. Using the filtration and
the long exact sequence of cohomological groups it is immediate to
prove that any torsor under a filtered group scheme is locally Zariski
trivial. So its essential dimension is zero.

The second statement is obtained using \reftext{Lemma~\ref{lemm:edsottogruppi}}.
\end{proof}

\begin{lemm}
\label{lem3}Let us suppose that $K$ has characteristic $p$. Let $G$ be a model (not
necessarily finite) over $S$ of a finite infinitesimal unipotent group
scheme over $K$. Then there exists a central decomposition series
\begin{equation*}
1 = G_{0} \subseteq G_{1} \subseteq \dots \subseteq G_{r} = G
\end{equation*}
of $G$ made by  group schemes over $S$, such that each successive
quotient $G_{i}/G_{i-1}$ is a subgroup scheme of ${\mathbb{G} _{
\mathrm{a} }}_{,S} $ of order $p$.
\end{lemm}
\begin{proof}
We remark that the generic fiber is necessarily of order $p^{n}$. By
\cite[Proposition IV \S 2, 2.5]{DG} the result is true on the generic
fiber. Taking the schematic closure we obtain a central decomposition
series where the quotients are models of subgroups of
$\mathbb{G}_{a,K}$. By \reftext{Lemma~\ref{lemm:modelssubgroupsGa}} we are
done.
\end{proof}

\begin{lemm}
\label{lemm:unipotentcohomologyzero}
Let us suppose that $K$ has characteristic $p$. Let $G$ be a model (not
necessarily finite) over $S$ of a finite commutative unipotent
infinitesimal group scheme. Then for any affine scheme $X$ over $S$ we
have that $H^{j}(X,G)=0$ if $j\ge 2$.
\end{lemm}
\begin{proof}
By the previous lemma we are reduced to proving it in the case of a
closed subgroup scheme of $\mathbb{G}_{a,S}$ of order $p$. First of all
we remark that $\mathbb{G}_{a,S}/G$ is isomorphic to $\mathbb{G}_{a,S}$.
In fact for the fibers this follows by
\cite[IV, \S 2 Proposition 1.1]{DG}. Therefore $\mathbb{G}_{a,S}/G$ is
a smooth model of $\mathbb{G}_{a,K}$ with special fiber isomorphic to
$\mathbb{G}_{a,k}$, in particular connected. Then it isomorphic to
$\mathbb{G}_{a,S}$ by \cite[Theorem 2.2]{WW}. Now it is well known that
$H^{i}(X,\mathbb{G}_{a,S})=0$, if $i\ge 1$, so the wanted result easily
follows.
\end{proof}

\begin{lemm}
\label{lemm:epiincohomology}
Let $X=\operatorname{Spec}(A)$ be an affine faithfully flat scheme over
$S$. If $f:G_{1}\to G_{2}$ is an epimorphism of  $S$-group schemes
such that kernel has unipotent infinitesimal generic fiber and it is
central in $G_{1}$ then $H^{1}(X,G_{1})\to H^{1}(X,G_{2})$ is
surjective.
\end{lemm}
\begin{proof}

Let $P\to X$ be a $G_{2}$-torsor. By \cite[IV Proposition 2.5.8]{Gi},
the gerbe of all liftings of $P\to X$ is banded by the group scheme
$_{P} \ker f$, i.e. the group scheme obtained by $\ker f$ twisting by
$P\to X$ with $G_{2}$ acting by conjugation. Since $\ker f$ is central
in $G$ then $_{P} \ker f\simeq \ker f$. Again by \cite[IV Proposition
2.5.8]{Gi}, the torsor $P\to X$ is in the image of the map $H^{1}(X,G
_{1})\to H^{1}(X,G_{2})$ if and only if the above gerbe is trivial. But
gerbes over $Sch/X$ banded by a group $G$ are classified by
$H^{2}(X,G)$. Now the generic fiber of $\ker f$ is unipotent
infinitesimal so by \reftext{Lemma~\ref{lemm:unipotentcohomologyzero}} we have
that $H^{2}(X,\ker f)=0$. We have so proved that $H^{1}(X,G_{1})
\to H^{1}(X,G_{2})$ is surjective.
\end{proof}

Here we give a sort of generalization of \cite[Theorem 1.4]{TV} over a
discrete valuation ring for finite group schemes. We first prove the
following Lemma, which is in the counterpart of \cite[Lemma 3.4]{TV} in
a less general form.

\begin{lemm}
\label{lemm:edsuccessionekernelunipotent}
Let 
\begin{equation*}
1\longrightarrow G_{1}\longrightarrow G\longrightarrow G_{2}\longrightarrow
1
\end{equation*}
be an exact sequence of group schemes over $S$ such that $G_{1}$ is central in $G$ and the
generic fiber of $G_{1}$ is unipotent infinitesimal. Then
\begin{equation*}
\operatorname{ed}_{R}(G)\le \operatorname{ed}_{R}({G_{1}})+
\operatorname{ed}_{R}(G_{2}).
\end{equation*}
\end{lemm}

\begin{rema}
\label{rem:lemmadebole}
In fact this Lemma, over a field, is weaker than \cite[Lemma 3.4]{TV}
also in the finite case. The point here is that to have a similar
statement one should involve twisted forms of $G_{1}$, defined over a
scheme $Y$ of dimension maybe greater than $1$. We do not have a
decomposition for these group schemes and so we can not apply
d\'{e}vissage arguments to reduce to the case of subgroup schemes of
$\mathbb{G}_{a}$. Moreover we need that the generic fiber of
$G_{1}$ is infinitesimal otherwise we do not know if we can obtain a
filtration with quotients which are subgroups of $\mathbb{G}_{a,S}$.
More precisely one should prove that any model of a simple \'{e}tale
unipotent group scheme is contained in $\mathbb{G}_{a,S}$
\end{rema}
\begin{proof}
Let $f:P\to X$ be a $G$-torsor. Let us consider the $G_{2}$-torsor
$f_{1}:P_{1}=P/G_{1} \to X$. Then there exists a $G_{2}$-torsor
$f_{2}:P_{2}\to Y$ which is a weak $S$-compression of $f_{1}$ and such
that $\operatorname{ed}_{S} f_{2}=\dim_{S} Y$. Up to take the schematic
closure in $X$, since we are over a discrete valuation ring, we can
suppose that is an $S$-compression, not only weak. So $X\dashrightarrow
Y$ is $S$-dominant. Now by \reftext{Lemma~\ref{lemm:epiincohomology}} there
exists a $G$-torsor $g:Q\to Y$ such that $Q/G_{1}\to Y$ is isomorphic
to $f_{2}$ as $G_{2}$-torsors. If $U$ is an $S$-dense open subscheme
where $X\dashrightarrow Y$ is defined then $f_{U}:P_{U}\to U$ and
$g_{U}:Q_{U}\to U$ have the same image in $H^{1}(U,G_{2})$. Now, by
\cite[III, Proposition 3.4.5]{Gi}, we have that $H^{1}(U,G_{1})$ acts
transitively on the fibers of $H^{1}(U,G)\to H^{1}(U,G_{2})$. So let
$h_{1}:Z\to U$ be a $G_{1}$-torsor such that its class sends the class
of $g_{U}$ in the class of $f_{U}$. Let $h_{2}:Z'\to U'$ be an
$S$-compression of $h_{1}$ with $\dim_{S} U'=\operatorname{ed}_{S} h
_{1}$. Now take $[(h_{2})_{U'\times_{S} Y}]\cdot [g_{U'\times_{S} Y}]
\in H^{1}(U'\times_{S} Y,G)$, where $\cdot $ is the action of
$H^{1}(U\times_{S}Y,G_{1})$ over $H^{1}(U\times_{S} Y,G)$. The
associated torsor is a weak $S$-compression of $f$. So
\begin{equation*}
\operatorname{ed}_{S} f\le \dim_{S} U'+\dim_{S} Y= \operatorname{ed}
_{S} h_{1} +\operatorname{ed}_{S} f_{2} \le \operatorname{ed}_{S} G
_{1}+\operatorname{ed}_{S} G_{2}.\qedhere
\end{equation*}
\end{proof}

\begin{theo}
\label{theo:edminorepn}
Suppose that $R$ has characteristic $p$. Let $G$ be a finite group
scheme of order $p^{n}$ over $S$, such that
\begin{itemize}
\item[(i)] $G_{K}$ is the product of an unipotent infinitesimal group
scheme and diagonalizable group scheme;
\item[(ii)]the closure of the diagonalizable part is Kummer.
\end{itemize}
Then the essential dimension is less than or equal to $n$.
\end{theo}

\begin{rema}
The condition $(\mathit{ii})$ is conjecturally empty, i.e. any model of
a diagonalizable group is Kummer. For instance this is the case for
group scheme of order $p^{2}$ \cite{To2} and some evidences can be
found in \cite{MRT1}. After the proof we will comment on the
condition~$(i)$.
\end{rema}

\begin{proof}
If $G$ is a Kummer group scheme then we have the result by
\reftext{Lemma~\ref{lemm:edKummer}}. Now we suppose that the generic
fiber is infinitesimal unipotent. We prove the result by induction on
$n$. If $n=0$ it is clear. Now we suppose $n\ge 1$. By \cite[IV
Proposition~2.5]{DG} there exists a central subgroup of the generic
fiber of order $p$ and isomorphic to a subgroup of $ \mathbb{G}_{a,K}$.
Take the schematic closure $H$. It is isomorphic to a closed subgroup
scheme of $\mathbb{G}_{a,S}$ finite over $S$ by
\reftext{Lemma~\ref{lemm:modelssubgroupsGa}}. Moreover it is central in
$G$. Therefore its essential dimension is less than or equal to $1$ by
\reftext{Lemma~\ref{lemm:edsottogruppi}} and
\reftext{Proposition~\ref{prop:edmodelliGa}}. Now by
\reftext{Lemma~\ref{lemm:edsuccessionekernelunipotent}} we have that
\begin{equation*}
\operatorname{ed}_{S} G\le \operatorname{ed}_{S}H+\operatorname{ed}
_{S} G/H\le 1+\operatorname{ed}_{S}G/H.
\end{equation*}
So we can conclude by induction.

Now we consider the general case. We argue by induction, again. Let
$n\ge 1$. We suppose the statement is true for group schemes of order
$n-1$ and we prove it for $n$. We call $G_{u,K}$ and $G_{d,K}$,
respectively, the unipotent and the diagonalizable part of the generic
fiber. We can suppose that $G_{u,K}$ is nontrivial by what we proved
previously. By \cite[IV Proposition 2.5]{DG} there exists a central
subgroup of $G_{u,K}$ of order $p$ and isomorphic to a subgroup of
$\mathbb{G}_{a,K}$. Take the schematic closure $H$. It is isomorphic to
a closed subgroup scheme of $\mathbb{G}_{a,S}$ finite over
$S$ by \reftext{Lemma~\ref{lemm:modelssubgroupsGa}}. Therefore its essential
dimension is less than or equal to $1$ by \reftext{Lemma~\ref{lemm:edsottogruppi}} and \reftext{Proposition~\ref{prop:edmodelliGa}}.
Moreover $H$ is also central in $G$ since $G_{K}=G_{u,K}\times G_{d,K}$.
Now by \reftext{Lemma~\ref{lemm:edsuccessionekernelunipotent}} we have that
\begin{equation*}
\operatorname{ed}_{S} G\le \operatorname{ed}_{S}H+\operatorname{ed}
_{S} G/H\le 1+\operatorname{ed}_{S}G/H.
\end{equation*}
So we can conclude by induction.
\end{proof}

\begin{rema}
We observe that in \cite[Theorem 2.2]{TV} the above result is proved
over a field without the hypothesis $(i)$. The main point is that \reftext{Lemma~\ref{lemm:edsuccessionekernelunipotent}} is weaker than \cite[Lemma
3.4]{TV} as observed in the \reftext{Remark~\ref{rem:lemmadebole}}.
\end{rema}

We remark that in \reftext{Corollary~\ref{coro:quasispecialisopen}} we proved
that if $G$ is a model of an almost special group scheme over $K$ then
\begin{equation*}
\operatorname{ed}_{k} G_{k}\ge \operatorname{ed}_{K}{G_{K}}.
\end{equation*}
If $K$ has characteristic $p>0$, this applies, for instance, for models
of diagonalizable group schemes with finite part of order a power of
$p$, since smooth diagonalizable group schemes and diagonalizable
$p$-group schemes are almost special \cite[Example~4.4]{TV}. We remark
that in general the special fiber is not diagonalizable but unipotent.
If $G$ is a finite group scheme (flat) over $S$, with $S$ complete, and
the order of $G$ is not divisible by $p$ then $\operatorname{ed}_{K} G
_{K}\ge \operatorname{ed}_{k} G_{k}$ (see \cite[Theorem 5.11]{BRV1}).
There is equality if $K$ has characteristic~$p$ (see
\cite[Corollary~5.12]{BRV1}). In the examples after the next Corollary
we will see that for finite  $p$-group schemes over $S$ a general
result can not exist. Now we give a more precise result for models of
group schemes of height ${\le} 1$.

\begin{coro}
If $K$ is of positive characteristic and $G_{K}$ is a trigonalizable group scheme of
height ${\le} 1$, i.e. killed by Frobenius, and order $n$ then
\begin{equation*}
\operatorname{ed}_{k} G_{k}= \operatorname{ed}_{K}{G_{K}}=n.
\end{equation*}
If moreover $G_{K}$ satisfies conditions of \reftext{Theorem~\ref{theo:edminorepn}} then both quantities are equal to $
\operatorname{ed}_{S} G$.
\end{coro}
\begin{proof}
If $G_{K}$ is killed by Frobenius then also $G$ is clearly killed by
Frobenius. So the special fiber is of height ${\le} 1$. Therefore the
essential dimension of fibers is equal to $n$, by the proof of
\cite[Corollary 4.5]{TV}. Finally if we can apply \reftext{Theorem~\ref{theo:edminorepn}} then $\operatorname{ed}_{S} G\le n$, where
$n$ is order of $G$ and we are done, using \reftext{Corollary~\ref{coro:edandbasechange}}.
\end{proof}

\begin{exam}
We give here examples which show that in general everything can happen
if $G$ is a finite $p$-group scheme over $S$.
\begin{itemize}
\item[(i)] Let us suppose that the characteristic of $K$ is zero and
$K$ contains a primitive $p^{2}$-th root of unity. Then $(\mathbb{Z}/p
^{2}\mathbb{Z})_{K}\simeq \mu_{p^{2},K}$ so its essential dimension is
$1$. On the other hand $\operatorname{ed}_{k}(\mathbb{Z}/p^{2}
\mathbb{Z})_{k}=2$ by \cite[Proposition 7.10]{BF}. So
\begin{equation*}
\operatorname{ed}_{k}(\mathbb{Z}/p^{2}\mathbb{Z})_{k}>
\operatorname{ed}_{K}(\mathbb{Z}/p^{2}\mathbb{Z})_{K}
\end{equation*}

Moreover $(\mathbb{Z}/p\mathbb{Z})^{2}_{K}\simeq \mu_{p,K}^{2}$ (just
a primitive $p$-th root is needed). Therefore its essential dimension
is $2$ by \cite[Corollary 3.9]{BF}. On the other hand, if $k$ infinite,
$\operatorname{ed}_{k}(\mathbb{Z}/p\mathbb{Z})^{2}_{k}=1$ (see
\reftext{Proposition~\ref{Proposition:finitefield}}). So, if $k$ is infinite,
\begin{equation*}
\operatorname{ed}_{K}(\mathbb{Z}/p\mathbb{Z})^{2}_{K}>
\operatorname{ed}_{k}(\mathbb{Z}/p\mathbb{Z})^{2}_{k}
\end{equation*}
\item[(ii)] Let us suppose that the characteristic of $K$ is $p$. As
noted above if $G_{K}$ is diagonalizable then $\operatorname{ed}_{k} G
_{k} \ge \operatorname{ed}_{K} G_{K}$. Strict inequality can happen.

Now take $G=\operatorname{Spec}(A)$ with $A=R[T_{1},T_{2}]/(T_{1}^{p}-T
_{1},T_{2}^{p}-\pi^{(p-1)}T_{2})$, with $\pi$ an uniformizer of $R$.
We define the multiplication by
\begin{align*}
T_{1}
&\mapsto T_{1}\otimes 1+1\otimes T_{1}
\\
T_{2}
&\mapsto T_{1}\otimes 1+1\otimes T_{1}+\pi \sum_{i=1}^{p-1} \frac{
\binom{p}{i}}{p}T_{2}^{i}\otimes T_{2}^{p-i}
\end{align*}
This group schemes is such that its N\'{e}ron blow-up in the subgroup
$(\mathbb{Z}/p\mathbb{Z})_{k}$ of the special fiber is isomorphic to
$(\mathbb{Z}/p^{2}\mathbb{Z})_{R}$. One can easily see that
$G_{k}\simeq \alpha_{p,k}\times (\mathbb{Z}/p\mathbb{Z})_{k}$. So it is
isomorphic to a subgroup scheme of $\mathbb{G}_{a,k}$, hence its
essential dimension is $1$. On the other hand $\operatorname{ed}_{K} G
_{K}=2$ since $G_{K}$ is isomorphic to
$(\mathbb{Z}/p^{2}\mathbb{Z})_{K}$. This shows that
\begin{equation*}
\operatorname{ed}_{K} G_{K}> \operatorname{ed}_{k} G_{k}.
\end{equation*}
\end{itemize}
\end{exam}

\bibliographystyle{amsalpha}
\bibliography{essdimdvr_weak4}
\end{document}